\documentclass{amsart}

\usepackage{amssymb}
\usepackage{amsmath}
\usepackage{amsthm}
\usepackage{amsbsy}
\usepackage{bm}
\usepackage{hyperref}
\date{\today}
\usepackage{cite}
\usepackage{array}
\usepackage{enumerate}
\usepackage{color}

\newcommand{\CZ}{\mathcal{Z}}

\theoremstyle{theorem}
    \newtheorem{theorem}{Theorem}
    \newtheorem{lemma}[theorem]{Lemma}
    
    \newtheorem{corollary}[theorem]{Corollary}

\theoremstyle{definition} 
    \newtheorem{definition}[theorem]{Definition}
    
    \newtheorem{result}[theorem]{Result}
    \newtheorem{remark}[theorem]{Remark}
    \newtheorem{example}[theorem]{Example}
    \newtheorem{exercise}[theorem]{Exercise}

\def\C{\mathbb{C}}

\def\CL{\mathcal{L}}

\def\R{\mathbb{R}}

\def\f{{\bf f}}

\def\l{\left}
\def\r{\right}
\def\<{\langle}
\def\>{\rangle}

\newcommand{\E}{\mbox{\bf E}}

\def\P{{\bf P}}

\newcommand\Tr{{\mbox{Tr}}}

\newcommand\mnote[1]{} 
\newcommand\be{\begin{equation*}}

\newcommand\ee{\end{equation*}}

\newcommand\ben{\begin{equation}}
\newcommand\een{\end{equation}}
\newcommand\bes{\begin{eqnarray*}}
\newcommand\ees{\end{eqnarray*}}

\newcommand\bex{\begin{exercise}}
\newcommand\eex{\end{exercise}}
\newcommand\beg{\begin{example}}
\newcommand\eeg{\end{example}}
\newcommand\benu{\begin{enumerate}}
\newcommand\eenu{\end{enumerate}}
\newcommand\beit{\begin{itemize}}
\newcommand\eeit{\end{itemize}}
\newcommand\berk{\begin{remark}}
\newcommand\eerk{\end{remark}}
\newcommand\bdefn{\begin{defintion}}
\newcommand\edefn{\end{definition}}
\newcommand\bthm{\begin{theorem}}
\newcommand\ethm{\end{theorem}}
\newcommand\bprf{\begin{proof}}
\newcommand\eprf{\end{proof}}
\newcommand\blem{\begin{lemma}}
\newcommand\elem{\end{lemma}}

\newcommand{\var}{\mbox{\rm Var}}

\newcommand{\sm}{{\raise0.3ex\hbox{$\scriptstyle \setminus$}}}

\def\l{\left}
\def\r{\right}

\def\CHI{\mathchoice%
{\raise2pt\hbox{$\chi$}}%
{\raise2pt\hbox{$\chi$}}%
{\raise1.3pt\hbox{$\scriptstyle\chi$}}%
{\raise0.8pt\hbox{$\scriptscriptstyle\chi$}}}
\def\smalloplus{\raise1pt\hbox{$\,\scriptstyle \oplus\;$}}

\allowdisplaybreaks

\newtheorem{cor}[theorem]{Corollary}
\title[CLT for random matrices with a variance profile]{Linear eigenvalue statistics of random matrices with a variance profile}

\author{Kartick Adhikari}
\address{Department of Electrical Engineering\\
        Technion - Israel Institute of Technology\\
        Haifa, 32000, Israel}

\email{kartickmath [at] gmail.com}

\author{Indrajit Jana}
\address{Department of Mathematics\\
Temple University\\
Philadelphia, PA, USA, 19122}

\email{ijana [at] temple.edu}

\author{Koushik Saha}
\address{Department of Mathematics\\
        Indian Institute of Technology Bombay\\
         Powai, Mumbai  400076, India}

\email{koushik.saha [at] iitb.ac.in}

\date{\today}
\thanks{The work of Koushik Saha is partially supported by Seed grant of IIT Bombay. The research of Kartick Adhikari supported  in part by Israel Science Foundation, Grant 2539/17 and Grant 771/17,  and by SERB (reference no. PDF/2016/001601) India.}
\begin{document}
\maketitle
\begin{abstract}
We give an upper bound on the total variation distance between the    
linear eigenvalue statistic, properly scaled and centred, of a random matrix with a variance profile and the standard Gaussian random variable. The second order Poincar\'e inequality type result introduced  in \cite{chatterjee2009fluctuations} is used to establish the bound.  Using this bound we prove  Central limit theorem for linear eigenvalue statistics of random matrices with different kind of variance profiles. We re-establish some existing results on fluctuations of linear eigenvalue statistics of some well known random matrix ensembles by choosing appropriate variance profiles.
\end{abstract}

\section{Introduction} 
Let $\lambda_1,\lambda_2,\ldots, \lambda_n$ be the eigenvalues of an $n\times n$ matrix $M$ with real or complex entries. The {\it linear eigenvalue  statistic} of $M$ is a function of the form
$$
\sum_{k=1}^{n}f(\lambda_k)
$$
where $f$ is a fixed function. The function $f$ is known as the {\it test function}. 
In this article we study the fluctuations of linear eigenvalue statistics  of random
matrices of the following form
\begin{equation}\label{definition:variance profile}
Y=A \circ X,
\end{equation}
where $\circ$ is the matrix Hadamard product, $A$ is an $n\times n$ deterministic matrix with non-negative entries $a_{ij}$ and $X$ is an $n\times n$ random matrix. 
In literature $A$ is referred as the {\it standard deviation profile}, and  $A\circ A$ is  referred as the {\it variance profile}. We call $Y$ as a random matrix with a variance profile $A\circ A$.  If $X$ is symmetric (Hermitian) and $A$ is symmetric, then $Y$ is symmetric (Hermitian) with a variance profile $A\circ A$.

Random matrices  and random matrices with  variance profiles  have been used in different areas of sciences, for instance, in ecology to study the stability of an ecological system with differnt species, and in  neuroscience to model networks. For an overview, we refer to \cite{may1972will}, \cite{allesina2015stability}, \cite{allesina2015predicting}, \cite{cohen_newman_1985}, \cite{somopolinsky1988}, \cite{aljadeff_renfrew_stern_2015}, \cite{aljadeff2015transition} and \cite{rajan2006eigenvalue}. 

In recent years there have been increasing interest to study random matrices with a variance profile. For results on Hermitian matrices with a variance profile, see \cite{dimitri_1996}, \cite{girko_vol1_2001}, \cite{guionnet_2002}, \cite{anderson2006clt}, \cite{hachem_loubaton_najim_2007}, \cite{jana2014}, \cite{li2013central} and  \cite{shcherbina2015}. 
Recently  Cook et. al. \cite{cook2016limiting} have studied the limiting spectral distribution of non-Hermitian  random matrices with different types of variance profile matrices. In this article, our goal is to study the fluctuations of  linear eigenvalue statistics of random  matrices with different type of variance profiles for  polynomial test functions.  We investigate the convergence of the fluctuations of linear eigenvalue statistics in total variation norm. 

The literature on linear eigenvalue statistics is quite large. To the best of our knowledge, 
the fluctuations of linear eigenvalue statistics  was first considered  by Arharov \cite{arharov} in 1971 for sample covariance matrices. 
In 1982, Jonsson \cite{jonsson} proved the Central limit theorem
(CLT) type results  of linear eigenvalue statistics for Wishart matrices using method of moments.   
In last three decades the fluctuations of linear eigenvalue statistics have become one of the popular field  of research in random matrix theory.  There are several results on the fluctuations of linear eigenvalue statistics of random matrix ensembles of different type.  
To get an overview on the  results on  Wigner and sample covariance matrices,   we refer the readers to   \cite{johansson1998}, \cite{soshnikov1998tracecentral}, \cite{bai2004clt}, \cite{lytova2009central}, \cite{shcherbina2011central}, \cite{sosoe_wong_2013} and the reference there in.  For   band and sparse symmetric random matrices, see  \cite{anderson2006clt},  \cite{jana2014}, \cite{li2013central},  \cite{shcherbina2015}. For Toeplitz, Hankel and circulant matrices, see  \cite{chatterjee2009fluctuations}, \cite{liu2012}, \cite{adhikari_saha2017} and \cite{adhikari_saha2018}, and for non-Hermitian matrices, see \cite{rider_silverstein_2006}, \cite{nourdin_peccati_2010} and \cite{rourke_renfrew_2016}.

In this article, we derive a simple  bound  on the total variation distance between the linear fluctuations, for the polynomial test functions, of $A\circ X$ and the standard Gaussian variable. The bound is given  
 in terms of the entries of $A$, the degree of the polynomial and the dimension of $X$. see Theorem \ref{thm: Non-symmetric matrix - Total variation}. Using Theorem \ref{thm: Non-symmetric matrix - Total variation}, we establish the Central limit theorem for the linear eigenvalue statistics of $A\circ X$ with different kind of variance profile matrices. For instance, in Corollary \ref{cor: for Erdos - Renyi graph}   we studied the fluctuation of $A\circ X$ when $A$ is an adjacency matrix of an Erd\"os-R\'enyi random graph. We re-establish some  existing results on the fluctuations of linear eigenvalue statistics of random matrices very easily with the appropriate choices of $A$. For example, in Corollary \ref{cor:xx'}, we showed that the fluctuation of linear eigenvalue statistics of sample covariance matrix is Gaussian. In Corollary \ref{cor:productofmatrices}, we established that the fluctuation of linear eigenvalue statistics of finite product of independent copies of i.i.d. matrices is Gaussian. We also derived the fluctuations of linear eigenvalue statistics for diagonal, anti-diagonal and sparse random matrices using Theorem \ref{thm: Non-symmetric matrix - Total variation}. see  Corollary \ref{iid_variance_band}, Corollary \ref{cor_iid_sparse} and Remark \ref{anti-diagonal}.

Now, we  briefly outline the rest of the article. In Section \ref{notation}, we introduce the notations frequently used in this article. In Section \ref{main result}, we introduce the assumptions on the  test functions, on the  entries of $X$, and  state our main result, Theorem \ref{thm: Non-symmetric matrix - Total variation}. In Section \ref{sec:example}, we derive some new results and re-establish some existing results on fluctuations of linear eigenvalue statistics  from Theorem \ref{thm: Non-symmetric matrix - Total variation}. In Section \ref{proof of main result}, we give the proof of Theorem \ref{thm: Non-symmetric matrix - Total variation} and in Section \ref{preliminary lemmas}, we collect the results needed to prove Theorem \ref{thm: Non-symmetric matrix - Total variation}. In Section \ref{sec: Appendix}, we give definitions of some variance profiles  to make this article self contained.

\subsection{Notations}\label{notation} Here we introduce some basic definitions and  notations used in this article.
\begin{enumerate}[(i)]
	\item Let $\mathcal{B}(\R)$ be the Borel sigma algebra on $\R$. Then we define  $d_{TV}(\mu, \nu):=\sup_{B\in \mathcal{B}(\R)}|\mu(B)-\nu(B)|$ the total variation distance between two probability measures $\mu$ and $\nu$ on the real line.
	\item Unless otherwise specified, all matrices are $n\times n$ square matrices with growing $n$. We suppress the subscript $n$ to avoid notational complexity. 
	\item Let $A=(a_{ij})_{n\times n}$ and $B=(b_{ij})_{n\times n}$ be two matrices of same dimension. Then $A\circ B$ denotes the standard Hadamard product, that is, $(A\circ B)_{ij}=a_{ij}b_{ij}$.
	\item $\{e_{1},e_{2},\ldots,e_{n}\}$ is the canonical basis of $\R^{n}$.
	\item $[n]=\{1,2,\ldots, n\}$.
	\item $I_k=\{(i_1,\ldots,i_k)\in[n]^k\; :\;  i_{p}\neq i_{q},\;\mbox{for all}\;p\neq q\}$.
	\item Let $\{a_{n}\}_{n}$ and $\{b_{n}\}_{n}$ be two sequence of non-negative real numbers. We write $a_{n}\lesssim b_{n}$ if there exists $c>0$ such that $a_{n}\leq cb_{n}$ for all $n$. We write $a_{n}\sim b_{n}$ if $\lim_{n\to\infty}(a_{n}/b_{n})=1$.
\end{enumerate}

\section{Main result}\label{main result}
 Let $P_k(x)=\sum_{i=0}^{k}c_ix^i$, where $c_{k}\neq 0$ and $c_{i}\in \R\cap [-\tau, \tau],\;\;\mbox{ for all} \; 1\leq i\leq k$, for some $k$-independent constant $\tau>0.$ In other words, the coefficients of the polynomial remain bounded even when the degree of the polynomial goes to infinity. In this article, we shall consider this type of polynomial test functions only. Now define
  \begin{align*}
 \CZ_{k}(A\circ X)=\frac{\Tr(P_k(A\circ X ))-\E[\Tr(P_k(A\circ X))]}{\sqrt{\var(P_k(A\circ X))}},
 \end{align*}	
where $A=(a_{ij})$ is an $n\times n$ deterministic matrix with non-negative $a_{ij}$ and $X=(X_{ij})$ is an $n\times n$ random matrix.  

In our results,  $X$ is either an \textit{i.i.d. random matrix} or a \textit{symmetric random matrix}. By i.i.d. random matrix we mean  $\{X_{ij}; \; i,j\geq 1\}$ are i.i.d. random variables, and by symmetric random matrix we mean  $X_{ji}=X_{ij}$ and $\{X_{ij}; \; 1\leq i\le j \}$ are i.i.d. random variables.

We show that $\CZ_{k}(A\circ X)$ converges to the standard Gaussian distribution in total variation norm when $X$ is an i.i.d. or a symmetric random matrix and  $X_{ij}$ belongs to a specific class of distributions, namely, $\mathcal L(c_1,c_2)$ for some $c_1, c_2>0$.   

For each $c_1,c_2>0$, $\mathcal L(c_1,c_2)$ is a class of probability measures on $\mathbb R$ that arise as laws of random variables like $u(Z)$, where $Z$ is a standard Gaussian random variable and $u$ is a twice continuously differentiable function such that for all $x\in \mathbb R$
$$
|u'(x)|\le c_1 \mbox{ and } |u''(x)|\le c_2.
$$ 
For example, the standard Gaussian random variable is in $\mathcal L(1,0)$. The uniform distributed random variable on $[0,1]$ is in $\mathcal L((2\pi)^{-1/2}, (2\pi e)^{-1/2})$.  To the best of our knowledge, the linear eigenvalue statistics with this class of random variables was first considered in \cite{chatterjee2009fluctuations}.

\begin{theorem}\label{thm: Non-symmetric matrix - Total variation}
	 Let $X=(X_{ij})$ be an $n\times n $  i.i.d.  (respectively, symmetric) random  matrix, where $X_{ij}$ are symmetric random variables with variance one  and  $X_{ij}\in \mathcal L(c_1,c_2)$ for some $c_1,c_2>0$. 
	 Let  $A=(a_{ij})$ be an $n\times n$ deterministic (respectively, symmetric) matrix with non-negative entries and $$b_n=\max_{i,j}\Big\{\sum_{k=1}^na_{ik}^2, \sum_{k=1}^na_{kj}^2,\log n\Big\}.$$ Then
	\begin{align}\label{eqn:non symmetric:TV}
	d_{TV}\l(\CZ_{k}(A\circ X),Z\r) \lesssim \frac{(\max \{|a_{ij}|\})^2k^5\sqrt {n}b_n^{k-1}}{\sum_{I_k}a_{i_1i_2}^2a_{i_2i_3}^2\cdots a_{i_ki_1}^2},
	\end{align}  
	where $Z$ is a standard Gaussian random variable and $I_k=\{(i_1,\ldots,i_k)\in[n]^k\; :\;  i_{p}\neq i_{q},\;\mbox{for all}\;p\neq q\}$.
\end{theorem}
Observe that, when the degree of the polynomial $k$ is fixed, the factor $k^5$ in  \eqref{eqn:non symmetric:TV} can be absorbed in the implying constant  of `$\lesssim$'. We kept $k^5$ in the expression as  we will change the degree of the polynomial with $n$. 

To prove the theorem we use second order Poincar\'e inequality type result introduced in \cite[Theorem 3.1]{chatterjee2009fluctuations}. Suppose $X=(X_1,X_2,\ldots,X_n)$ is a vector of independent standard Gaussian random variables and $g:\mathbb R^n\to \mathbb R$ is smooth function. Then Gaussian Poincar\'e inequality says that $\mbox{Var}(g(X))\leq \E\|\nabla{g(X)}\|$, that is, if  $\|\nabla{g(X)}\|$ is small then $g(X)$ has small fluctuation. Now the second order Poincar\'e inequality  says if the second order derivatives of $g$ have good behaviour then $g(X)$ is close to Gaussian. We use this idea to  prove the  bound in \eqref{eqn:non symmetric:TV}. For that   
we  need to estimate a lower bound of the variance of $\CZ_{k}(A\circ X)$ and a non-asymptotic upper bound  for the norm of $A\circ X$. In Lemma \ref{lem:variance},  we show that the variance of $\CZ_{k}(A\circ X)$ is bounded below by $\sum_{I_k}a_{i_1i_2}^2a_{i_2i_3}^2\cdots a_{i_ki_1}^2$. In Lemma \ref{Lemma: Norm of the Matrix}, we show that the norm of $\frac{1}{\sqrt{b_{n}}}A\circ X$ is bounded almost surely. The main ingredients of the proof of Lemma \ref{Lemma: Norm of the Matrix}  are the concentration type inequalities and a sharp non-asymptotic bound on the expected norm of a symmetric random matrix derived in \cite[Corollary 3.5]{bandeira2016sharp}. 
The restriction on $b_n$ in Theorem \ref{thm: Non-symmetric matrix - Total variation}, that is, $b_n\geq \log n$  is required to establish the norm bound of $\frac{1}{\sqrt{b_{n}}}A\circ X$.

In the next section, using Theorem \ref{thm: Non-symmetric matrix - Total variation} we show that the fluctuations of the linear eigenvalue statistics  of $A\circ X$ for various types of variance profile matrices are asymptotically Gaussian.

\section{Applications of Theorem \ref{thm: Non-symmetric matrix - Total variation}}\label{sec:example}
There are various kind of variance profile matrices that have appeared in different branches of sciences. For a good overview on variance profile matrices, we refer the readers to a recent article by	Cook et. al. \cite{cook2016limiting}. In the following corollaries, we show that the  linear eigenvalue statistics of $A\circ X$ with different type of variance profile matrices converge to Gaussian distribution in total variation norm. We start with a basic variance profile matrix, namely, \textit{separable variance profile}.   

\begin{corollary}[Separable variance profile] Let $c>0$, and $v,w\in (c,1]^{n}$ be two deterministic vectors for each $n\geq 1$.  Consider
	$$
	A\circ A=vw^{T}
	$$
Suppose  $P_k$ is a polynomial of degree $k$ where $k=o(\log n)$ and $X$ is same as in Theorem \ref{thm: Non-symmetric matrix - Total variation}, then  $d_{TV}\l(\CZ_{k}(A\circ X),Z\r)\rightarrow 0$ as $n\rightarrow \infty$.
\end{corollary}

\begin{proof}
Note that $a_{ij}^2=v_{i}w_{j}$, for $1\le i,j\le n$. Therefore we have 
$$
\sum_{k=1}^na_{ik}^2=\sum_{k=1}^{n}v_{i}w_{k}\le n \mbox{ and } \sum_{k=1}^{n}a_{kj}^2\le n.
$$
Now, using Stirling's approximation we have
 \begin{align}\label{order of I_k}
 \frac{1}{n^{k}}|I_{k}|=\frac{n!}{n^{k}(n-k)!}\sim e^{-k}\left(1-\frac{k}{n}\right)^{-n+k}\sim 1,
 \end{align} 
 where the last asymptotic equality follows whenever $k=o(\sqrt{n})$.
Therefore,
$$
\sum_{I_k}a_{i_1i_2}^2a_{i_2i_3}^2\cdots a_{i_ki_1}^2\ge c^{2k}|I_k|=c^{2k}\Omega{(n^k)}.
$$	
Now from Theorem \ref{thm: Non-symmetric matrix - Total variation}, for $k=o(\log n)$, we have
	\begin{align*}
	d_{TV}\l(\CZ_{k}(A\circ X),Z\r)\lesssim \frac{k^5n^{k-1}\sqrt{n}}{c^{2k}n^k}=\frac{k^5}{c^{2k} \sqrt{n}}\to 0,\; \mbox{as $n\to \infty$}.
	\end{align*}
	Hence the result.
\end{proof}
In the following corollary, the variance profile matrix is constructed from a continuous positive real valued function $f(\cdot,\cdot)$ defined on $[0,1]^2$. This type of variance profile is known as  \textit{Sampled variance profile}. 

If $f(x,y)=g(x)h(y)$ where $g, h:[0,1]\to (0,\infty)$ are two continuous functions, then the corresponding  variance profile is known as \textit{Separable and sampled variance profile} (see  \cite{cook2016limiting}).

\begin{corollary}[Sampled variance profile]\label{Sampled variance profile}
Consider $A=(a_{ij})$ with $a_{ij}^2=f(\frac{i}{n},\frac{j}{n})$, where $f(\cdot,\cdot)$ is a positive continuous  on $[0,1]^2$ with $\int_{0}^{1}f(x,y)dy=1$ and $\int_0^1f(x,y)dx$ $=1$. Suppose $P_k$ is a polynomial of degree $k$ where $k=o(\log n)$ and $X$ is same as in Theorem \ref{thm: Non-symmetric matrix - Total variation}, then  $d_{TV}\l(\CZ_{k}(A\circ X),Z\r)\rightarrow 0$ as $n\rightarrow \infty$.
\end{corollary}

\begin{proof}
Since $f$ is a continuous function on $[0,1]^2$,	 there exists a positive constant $C$ such that $f\le C$. Therefore, for $1\le i,j\le n$, we have
	$$
\sum_{k=1}^na_{ik}^2\le Cn \; \mbox{ and } \sum_{k=1}^{n}a_{kj}^2\le Cn.
	$$
	Again, using \eqref{order of I_k} we get
	\begin{align*}
	\lim_{n\to \infty}\frac{1}{n^k}\sum_{I_k}a_{i_1i_2}^2a_{i_2i_3}^2\cdots a_{i_ki_1}^2&=\lim_{n\to \infty}\frac{1}{n^k}\sum_{I_k}f\left(\frac{i_1}{n},\frac{i_2}{n}\right)\cdots f\left(\frac{i_k}{n},\frac{i_1}{n}\right)
	\\&=\int_{0}^{1}\cdots \int_{0}^{1}f(x_1,x_2)f(x_2,x_3)\cdots f(x_k,x_1)dx_1\cdots dx_k
	\\&=M \mbox{ (say) }.
	\end{align*}
	Therefore from Theorem \ref{thm: Non-symmetric matrix - Total variation}, for $k=o(\log n)$ and $\epsilon>0$ such that $M-\epsilon>0$, we have
	\begin{align*}
	d_{TV}\l(\CZ_{k}(A\circ X),Z\r)\lesssim \frac{k^5C^{k}n^{k-1}\sqrt n}{n^k(M-\epsilon)}=\frac{k^5 C^{k}}{\sqrt n(M-\epsilon)}\to 0,\; \mbox{as $n\to \infty$}.
	\end{align*}
	Hence the result.
\end{proof}
\begin{remark}\label{anderson_zeitouni}
Anderson and Zeitouni \cite{anderson2006clt} considered $n\times n$ symmetric random matrix $Y$ with on or above diagonal terms are of the form $\frac{1}{\sqrt n}f(i/n,j/n) X_{ij}$ where $X_{ij}$ are zero mean unit variance i.i.d. random variables with all moments bounded and $f$ is a continuous function on $[0,1]^2$ such that $\int_0^1 f^2(x,y)dy=1$. They established CLT for  linear eigenvalue statistics of $Y$ with polynomial test functions. They used moment method, and nice combinatorial arguments inspired from \cite{furedi_komlos_1981}.  If the test function is continuously differentiable with polynomial growth, then they established the CLT for the random variables $X_{ij}$ which satisfy a Poincar\'e inequality with common constant $c$.
\end{remark}

In the previous corollary, the variance profile was constructed from a continuous function. In particular, $a_{ij}$s were allowed to take the zero values. However if $\{a_{ij}\}$ do not originate from a continuous function, we need to assume that $\{a_{ij};1\leq i,j\leq n\}$ are  bounded away from zero uniformly and the lower bound may depend on $n$. 
This is shown in the following corollary.

\begin{corollary}\label{assu:3}
	Let $A$ be an $n\times n$ matrix such that $\frac{1}{n^{\alpha}}\lesssim a_{ij}\le 1$ with $\alpha<\frac{1}{4}$ and 
	$
	cd_n\le \sum_{k=1}^na_{ik}^2, \sum_{k=1}^na_{kj}^2\le C d_n.
	$
	for some positive constants $c$ and $C$. If $d_n\ge \log n$, $P_k$ is a polynomial of  degree $k$ with $k=o(n^{\frac{1}{10}-\frac{2\alpha}{5}})$ and $X$ is as in Theorem \ref{thm: Non-symmetric matrix - Total variation}, then 
	$d_{TV}\l(\CZ_{k}(A\circ X),Z\r)\rightarrow 0$ as $n\rightarrow \infty$.
\end{corollary}

\begin{proof}
	Note that we have 
	\begin{align*}
		b_n=\max_{i,j}\big\{\sum_{k=1}^na_{ik}^2, \sum_{k=1}^na_{kj}^2, \log n\big\}\le Cd_n.
	\end{align*}
		For $1\le i_1,i_3\le n$,  we have
	$$
	\sum_{i_2}a_{i_1i_2}^2a_{i_2i_3}^2\gtrsim \frac{1}{n^{2\alpha}}\sum_{i_2}a_{i_1i_2}^2 \gtrsim \frac{cd_n}{n^{2\alpha}}.
	$$
	Using the last inequality and summing over the rest of the indices one by one, we have
	\begin{align*}
	\sum_{I_k}a_{i_1i_2}^2\cdots a_{i_ki_1}^2\gtrsim \frac{cd_n}{n^{2\alpha}}\sum_{I_k\backslash \{i_2\}}a_{i_3i_4}^2\cdots a_{i_ki_1} ^2\gtrsim \frac{cd_n}{n^{2\alpha}}n(cd_n)^{k-2}=\frac{(cd_n)^{k-1}}{n^{2\alpha-1}}.
	\end{align*}
	Therefore from  \eqref{eqn:non symmetric:TV} , for $\alpha<\frac{1}{4}$ and $k=o(\log n)$, we get
	$$
	d_{TV}\l(\CZ_{k}(A\circ X),Z\r) \lesssim \frac{k^5(Cd_n)^{k-1}\sqrt{n}n^{2\alpha-1}}{(cd_n)^{k-1}}=k^5(C/c)^{k-1}n^{2\alpha-\frac{1}{2}} \to 0,
	$$
	as $n\to \infty$. Hence the result.
\end{proof}

Now we shall consider the linear eigenvalue statistics of band random matrices.
The fluctuations of linear eigenvalue statistics of symmetric band random matrices are well studied in literature. Li and Soshnikov \cite{li2013central} considered symmetric periodic band matrices $Y=(Y_{ij})$ where 
$$Y_{ij}=Y_{ji}=\left\{
\begin{array}{ll}
\frac{1}{\sqrt b_n}X_{ij} &\text{if }\ \min\{|i-j|, n-|i-j|\}\leq b_{n},\\
0 & \mbox{otherwise}.
\end{array}
\right.$$ 
and $\{X_{ij};\; i\leq j\}$ are i.i.d. random variables. They established CLT for linear eigenvalue statistics of $Y$ when $X_{ij}$ satisfies Poincar\'e inequality with same constant $c$ and the test function $g$  has continuous bounded derivative, and $\sqrt n<< b_n<<n$ ($a_n<<b_n$ means $a_n/b_n\to 0$ as $n\to \infty$). Later the conditions on the band width $b_n$ and on the test function were improved  in  \cite{jana2014fluctuations} and \cite{shcherbina2015}.

Now observe that the periodic (or non-periodic) symmetric band random matrix can be seen as a symmetric random matrix with a variance profile. For example, let $A$ be a periodic (or non-periodic) band matrix with band length $b_n$, that is, 
\begin{equation}\label{periodic band}
a_{ij}=\left\{
\begin{array}{ll}
1 &\text{if }\ \min\{|i-j|, n-|i-j|\}\leq b_{n},\;\;(\mbox{or if }\ |i-j|\leq b_n)\\
0 & \mbox{otherwise}.
\end{array}
\right.\end{equation}
Then $A\circ X$ is a periodic (or non-periodic) symmetric band random matrix if $X$ is a symmetric random matrix. In the following corollary we show that the linear eigenvalue statistics of band random matrices, symmetric and non-symmetric both, converge to Gaussian distribution in total variation norm. 
\begin{corollary}\label{iid_variance_band}
Let $A$ be a periodic (or non-periodic) band matrix with band length $b_n$ as defined in \eqref{periodic band}.
 Suppose $b_n\ge \log n$,  $P_k$ is a polynomial of degree $k$ where $k=o( n^{1/10})$ and $X$ is same as in Theorem \ref{thm: Non-symmetric matrix - Total variation}, then 
	$$d_{TV}\l(\CZ_{k}(A\circ X),Z\r)\rightarrow 0\; \mbox{ as $n\rightarrow \infty$.}$$
\end{corollary}

\begin{proof}
If $A$ be a band matrix with band length $b_n$, so is $A\circ A$. Now it is easy to see that
$$\sum_{I_k}a_{i_1i_2}^2a_{i_2i_3}^2\cdots a_{i_ki_1}^2 \gtrsim nb_n^{k-1}.$$

The result follows from \eqref{eqn:non symmetric:TV}.
\end{proof}

In the next corollary  we obtain the fluctuation of linear eigenvalue statistics for sample covariance  matrix $XX^t$ using Theorem \ref{thm: Non-symmetric matrix - Total variation}. 
\begin{cor}\label{cor:xx'}
Let $X=(X_{ij})$ be an $n\times m$ random matrix with i.i.d. entries, where $X_{ij}$ are as in Theorem \ref{thm: Non-symmetric matrix - Total variation}. If $\frac{m}{n}\to c$ as $n\to \infty $ for some $c>0$. Then, for $k=o(\log n)$,
\begin{align*}
 d_{TV}(\CZ_{k}( XX^{t}),Z)\to 0,\;\;\;\text{ as $n\to\infty$,}
\end{align*}
where $P_k$ as in Theorem \ref{thm: Non-symmetric matrix - Total variation} and $X^{t}$ denotes the transpose of $X$.
\end{cor}

\begin{proof}
	With out loss of generality we assume that $n\le m$.
	Let $Y$ be a $(n+m)\times (n+m)$ symmetric matrix as in Theorem \ref{thm: Non-symmetric matrix - Total variation}. Observed that
	\begin{align*}
	A\circ Y=\l[\begin{array}{cc}
	0 & X \\ X^{t} & 0
	\end{array}\r],  \mbox{ where  }  A=\l[\begin{array}{cc}
	0 & A_{12} \\ A_{12}^t & 0
	\end{array}\r],
	\end{align*}
	and $A_{12}$ is an $n\times m$ matrix with all entries $1$. So we have
	\begin{align*}
		(A\circ Y)^2=\l[\begin{array}{cc}
		XX^t & 0 \\ 0 &X^{t}X
		\end{array}\r].
	\end{align*}
	Let $Q_{2k}(x)=P_k(x^2)$. Then $\Tr (Q_{2k}(A\circ Y))=2\Tr (P_k(XX^{t}))+(m-n)c_0$, where $c_0$ is the constant term in the polynomial $P_k$. Also observe that 
	\begin{align*}
		\CZ_{k}( XX^{t})&=\frac{\Tr (P_{k}(XX^{t}))-\E[\Tr (P_{k}(XX^{t}))]}{\sqrt{\var \Tr (P_{k}(XX^{t}))}}\\
		&=\frac{\Tr (Q_{2k}(A\circ Y))-\E[\Tr (Q_{2k}(A\circ Y))]}{\sqrt{\var \Tr (Q_{2k}(A\circ Y))}}=:\CZ_{2k}^Q( A\circ Y) \mbox{ (say)}.
	\end{align*}
	Note that in $A\circ Y$, $b_{n+m}=m$ and we have 
	\begin{align*}
		\sum_{I_{2k}}a_{i_1i_2}^2\cdots a_{i_{2k}i_1}^2 = \sum_{i_1\neq i_3\neq \cdots\neq i_{2k-1}=1}^n\sum_{i_2\neq i_4\neq \cdots\neq  i_{2k}=n+1}^{n+m}a_{i_1i_2}^2\cdots a_{i_{2k}i_1}^2 \sim (nm)^{k}.
	\end{align*}
	Therefore from Theorem \ref{thm: Non-symmetric matrix - Total variation}, for $k=o( \log n)$, we get 
	\begin{align*}
	d_{TV}(\CZ_{2k}^Q( A\circ Y), Z)\lesssim \frac{(2k)^5\sqrt {n+m} m^{2k-1}}{(nm)^{k}}\lesssim \frac{k^5}{\sqrt n}\sqrt{1+\frac{m}{n}} \l(\frac{m}{n}\r)^{k-1}\to 0, \mbox{ as $n\to \infty$}. 
	\end{align*}
Hence the result.	 
\end{proof}

In the next corollary we derive the fluctuations of linear statistics for the product of independent copies of iid random matrices. The fluctuation of linear eigenvalue statistics of a single i.i.d. matrix was considered in \cite{rider_silverstein_2006}. In a recent article \cite{coston2018gaussian}, Coston and O'Rourke have considered the fluctuations of linear eigenvalue statistics for product of $m$ independent copies of i.i.d. matrices with  analytic test functions.  
Here we derive that result with polynomial test functions  when the entries of the i.i.d. matrices are in $\mathcal L(c_1,c_2)$.

\begin{corollary}\label{cor:productofmatrices}
Let 	$P_k$ and $X$ be as defined in Theorem \ref{thm: Non-symmetric matrix - Total variation}. Let $X_1,\ldots, X_m$ be $m$ independent copies of $X$. Then, for fixed $m$ and  $k=o( n^{1/10})$,
	\begin{align*}
	 d_{TV}(\CZ_{k}( X_1\cdots X_m),Z)\to 0,\;\;\;\text{ as $n\to\infty$}.
	\end{align*}
\end{corollary}

\begin{proof}
For $m=1$, the proof is straight forward. For $m=2$, the idea of the proof is similar to the proof of Corollary \ref{cor:xx'}. The only difference is that here we  start with a $2n\times 2n$ i.i.d. random matrix $Y$ instead of a $2n\times 2n$ symmetric random matrix. Then 
\begin{align*}
	A\circ Y=\l[\begin{array}{cc}
	0 & X_1 \\ X_2 & 0
	\end{array}\r],  \mbox{ where  }  A=\l[\begin{array}{cc}
	0 & A_{12} \\ A_{21} & 0
	\end{array}\r]
	\end{align*}
 and $A_{12}$, $A_{21}$ are two $n\times n$ matrices with all entries $1$. Observe that $X_1,X_2$ are two independent $n\times n$ i.i.d. matrices. Now the rest of the proof is similar to Corollary \ref{cor:xx'}. We skip the detail.
 
Now we prove the result for $m=3$. For $m>3$, it can be proved in a similar way.  Let $Y$ be a $3n\times 3n$ dimensional i.i.d. matrix as in Theorem \ref{thm: Non-symmetric matrix - Total variation}.  Then 
	\begin{align*}
	A\circ Y=\l[\begin{array}{ccc}
	0 & X_1 & 0 \\ 0 & 0 & X_2\\ X_3 & 0 & 0
	\end{array}\r],  \mbox{ where  }  A=\l[\begin{array}{ccc}
	0 & A_{12} & 0 \\ 0 & 0 & A_{23}\\ A_{31} & 0 & 0
	\end{array}\r]
	\end{align*}
	and $A_{12}$, $A_{23}, A_{31}$  are three $n\times n$ matrices with all entries $1$. Observe that $X_1,X_2,X_3$ are three independent  $n\times n$ i.i.d. random matrices and
	\begin{align*}
		(A\circ Y)^3=\l[\begin{array}{ccc}
		X_1X_2X_3 & 0 & 0 \\ 0 & X_2X_3X_1 & 0 \\ 0 & 0 & X_3X_2X_1
		\end{array}\r].
	\end{align*}
Now,	 let $Q_{3k}(x)=P_k(x^3)$. Then $\Tr (Q_{3k}(A\circ Y))=3\Tr (P_k(X_1X_2X_3))$ and 
	\begin{align*}
	\CZ_{k}( X_1X_2X_3)&=\frac{\Tr (P_{k}(X_1X_2X_3))-\E[\Tr (P_{k}(X_1X_2X_3))]}{\sqrt{\var \Tr (P_{k}(X_1X_2X_3))}}\\
	&=\frac{\Tr (Q_{3k}(A\circ Y))-\E[\Tr (Q_{3k}(A\circ Y))]}{\sqrt{\var \Tr (Q_{2k}(A\circ Y))}}=:\CZ_{3k}^Q( A\circ Y) \mbox{ (say)}.
	\end{align*}
	Note that in $A\circ Y$, $b_{3n}=n$ and we have 
	\begin{align*}
	&\sum_{I_{3k}}a_{i_1i_2}^2\cdots a_{i_{3k}i_1}^2 \\=& \sum_{i_1\neq i_4\neq \cdots\neq i_{3k-2}=1}^n\sum_{i_2\neq i_5\neq \cdots\neq  i_{3k-1}=n+1}^{2n} \sum_{i_3\neq i_6\neq \cdots\neq i_{3k}=2n+1}^{3n} a_{i_1i_2}^2\cdots a_{i_{3k}i_1}^2 \sim n^{3k}.
	\end{align*}
	Therefore from Theorem \ref{thm: Non-symmetric matrix - Total variation}, for $k=o( n^{1/10})$, we get 
	\begin{align*}
	d_{TV}(\CZ_{3k}^Q( A\circ Y), Z)\lesssim \frac{(3k)^5\sqrt {3n} n^{3k-1}}{n^{3k}}\lesssim \frac{k^5}{\sqrt {n}} \to 0, \mbox{ as $n\to \infty$}. 
	\end{align*}
	Hence the result.	 
\end{proof}
In the last corollary we considered finite product of independent copies of i.i.d. random matrices. However, using the same technique one can study the fluctuation of linear eigenvalue statistics for product of rectangular random matrices with i.i.d. entries. Suppose  $X_1,X_2,\ldots,X_m$ are independent rectangular random matrices with i.i.d. entries of dimensions $n_1\times n_2,n_2\times n_3,\ldots,n_m\times n_1$, respectively. With out loss of generality we assume that $n_1=\min_{1\leq k\leq m}n_k$. Now following the idea of the proofs of Lemma \ref{cor:xx'} and Lemma \ref{cor:productofmatrices}, it is easy to show that the fluctuation of linear eigenvalue statistics of $X_1X_2\cdots X_m$ is Gaussian if  $\max_{1\leq k\leq m}n_k/n_1\to c<\infty$ as $n_1\to\infty$.
   
In the next corollary we consider  $A$ as an adjacency matrix of an Erd\"os-R\'enyi random graph $G(n,p_{n})$.
\begin{cor}\label{cor: for Erdos - Renyi graph}
	Let $X$ be a random matrix as defined in Theorem \ref{thm: Non-symmetric matrix - Total variation}, and $A$ be the adjacency matrix of a random Erd\"os-R\'enyi random graph $G(n,p_{n})$ independent of $X$, where $p_{n}\geq n^{-\gamma}$ for some $\gamma\in  [0,1/2)$. Let $P_{k}$ be a polynomial of degree $k$, where $k=o(\log n/\log (2/p_{n}))$. Then 
	\begin{align*}
	d_{TV}(\CZ_{k}(A\circ X),Z)\to 0,\;\;\;\text{almost surely as $n\to\infty$.}
	\end{align*}
	Here almost surely is with respect to the probability measure on $G(n,p_{n})$.
\end{cor}

The following lemma will be used in the proof of Corollary \ref{cor: for Erdos - Renyi graph}.
\begin{lemma}\label{lem:erdosreny}
	Let $A$ be the matrix as defined in Corollary \ref{cor: for Erdos - Renyi graph} with $p_{n}\geq n^{-\gamma}$ for some fixed $\gamma \in [0,1/2)$. Then
	\begin{align*}
	(i)\;\;&\frac{1}{(np_{n})^k}\sum_{I_k}a_{i_1i_2}^2\cdots a_{i_ki_1}^2 \to 1, \;\;\mbox{almost surely,  as $n\to \infty$},\\
	(ii)\;\;&\max_{i}\sum_{j=1}^{n}a_{ij}^{2}\leq (1+\epsilon_{n})np_{n},\;\;\text{almost surely, as $n\to\infty$},
	\end{align*}
	for $k=o(\log n/\log (2/p_{n}))$ and $\epsilon_{n}=O(n^{-\alpha}p_{n}^{-1})$ with $\alpha\in (\gamma, 1/2)$.
\end{lemma}
The following inequality will be used in the proof of Lemma \ref{lem:erdosreny}.

\begin{result}[Bounded difference inequality]\cite[Lemma 1.2]{mcdiarmid1989method}\label{re1}
	Let $Z=f(X_1,X_2,\ldots, X_n)$ be a function of independent random variables $X_1,\ldots, X_n$. Let $X_i'$ be an independent copy of $X_i, i=1,\ldots, n. $ Suppose $c_1, \ldots, c_n$ are constants such that for each $i$,
	$$
	|f(X_1,\ldots,X_{i-1},X_i', X_{i+1}, \ldots, X_n)-f(X_1,\ldots, X_n)|\le c_i \;\mbox{almost surely}.
	$$
	Then for any $t\ge 0$ we have 
	$$
	\P(Z-\E [Z]\ge t)\le \exp\l\{-\frac{2t^2}{\sum_{i=1}^{n}c_i^2}\r\}.
	$$
\end{result}

\begin{proof}[Proof of Lemma \ref{lem:erdosreny}]
The entries of $A$ can be thought as a vector of length $n(n+1)/2$, use dictionary order. Define 
\begin{align*}
	f(A)=\sum_{I_k}a_{i_1i_2}^2\cdots a_{i_ki_1}^2 .
\end{align*} 
Let $a_{ij}'$ be an independent copy of $a_{ij}$, for $1\leq i\leq j\leq n$. Then 
\begin{align*}
	|f(A)-f(A-a_{12}e_{1}e_{2}^{t}+a_{12}'e_{1}e_{2}^{t})|
	&=k|a_{12}^2-a_{12}'^2|\sum_{I_{k-2}'}a_{2i_1}^2a_{i_1i_2}^2\cdots a_{i_{k-1}1}^2
	\\&\le 2k|I_{k-2}'|\le 2kn^{k-2}.
\end{align*}
Similarly, for $1\le i, j\le n$, we have 
\begin{align*}
	|f(A)-f(A-a_{ij}e_{i}e_{j}^{t}+a_{ij}'e_{i}e_{j}^{t})|\le 2kn^{k-2}.
\end{align*}
Note that only $(i,j)$-th element $a_{ij}$ is replaced by $a_{ij}'$. Since $a_{i_1i_2},\ldots, a_{i_ki_1}$, for $(i_1,\ldots, i_k)\in I_k$, are independent random Bernoulli random variables with parameter $p_{n}$, we have 
\begin{align*}
	\E\left[\sum_{I_k}a_{i_1i_2}^2\cdots a_{i_ki_1}^2 \right]=|I_k|p_{n}^k.
\end{align*}
 Using Stirling's approximation, we have
  \begin{align*}
 \frac{1}{n^{k}}|I_{k}|=\frac{n!}{n^{k}(n-k)!}\sim e^{-k}\left(1-\frac{k}{n}\right)^{-n+k}\sim 1,
 \end{align*} 
 where the last asymptotic equality follows from $k=o(\sqrt{n})$. Now, using the Result \ref{re1}, for $\delta>0$, we get 
\begin{align*}
\P\l(\l|\frac{\sum_{I_k}a_{i_1i_2}^2\cdots a_{i_ki_1}^2 -\E[\sum_{I_k}a_{i_1i_2}^2\cdots a_{i_ki_1}^2 ]}{(np_{n})^k}\r|>\delta\r)\le \exp\left\{-\frac{\delta^2n^{2}p_{n}^{2k}}{2k^2}\right\}.
\end{align*}
According to the choices of $k$ and $p_{n}$, we have $np_{n}^{k}/k\geq \sqrt{n}/\log n$. Hence the result (i) follows from the Borel-Cantelli lemma.

To prove (ii), let us define
\begin{align*}
\Omega_{n}=\{A\in G(n,p_{n}): \sum_{i=1}^{n}a_{ij}^{2}\leq (1+\epsilon_{n})np_{n}, \mbox{ for all}\;i\},
\end{align*}
where $\epsilon_{n}=p_{n}^{-1}n^{-\alpha}$. Using union bound and result \ref{re1}, 
\begin{align*}
\P(\Omega_{n}^{c})&\leq n \P\left(\sum_{i=1}^{n}a_{1j}>(1+\epsilon_{n})np_{n}\right)\\
&\leq n\exp\{-\epsilon_{n}^{2}np_{n}^{2}/2\}.
\end{align*}
According the choice of $\epsilon_{n}$, and $p_{n}$, we have $\epsilon_{n}^{2}np_{n}^{2}\geq n^{1-2\alpha}$. Consequently, $\sum_{n=1}^{\infty}n\exp\{-\epsilon_{n}^{2}np_{n}^{2}\}<\infty$. Hence the result (ii) follows from the Borel-Cantelli lemma.
\end{proof}

\begin{proof}[Proof of Corollary \ref{cor: for Erdos - Renyi graph}]
	 By Theorem \ref{thm: Non-symmetric matrix - Total variation}, Lemma \ref{lem:erdosreny} almost surely for each $A$, we have
	\begin{align*}
		d_{TV}(\CZ_{k}(A\circ X),Z)&\le \frac{k^5\sqrt{n}b_n^{k-1}}{\sum_{I_k}a_{i_1i_2}^2\cdots a_{i_ki_1}^2}
		\leq \frac{k^{5}\sqrt{n}(1+\epsilon_{n})^{k-1}(np_{n})^{k-1}}{\sum_{I_k}a_{i_1i_2}^2\cdots a_{i_ki_1}^2}\\
		&= \frac{k^{5}(1+\epsilon_{n})^{k-1}}{p_{n}\sqrt{n}}\frac{(np_{n})^{k}}{\sum_{I_k}a_{i_1i_2}^2\cdots a_{i_ki_1}^2}.
	\end{align*}
If we choose $k=o(\log n/\log (2/p_{n}))$ , then the above converges to zero as $n\to\infty$. Hence the result.
\end{proof}
In the following corollary we consider the variance profiles which have an $m\times m$ block of ones along the diagonal, where $m$ is of the order of $n$. So this corollary  captures the  fluctuation of linear eigenvalue statistics  of a specific type of  (symmetric or non-symmetric) sparse random matrices which have a block of non-zero entries  of dimension $cn\times cn$ along the diagonal for some $c\in(0,1)$. For more results on fluctuations of symmetric sparse random matrices, we refer to \cite{shcherbina2015}.  
\begin{corollary}\label{cor_iid_sparse}
Let  $P_{k}$, $X$, $A$ be as in Theorem \ref{thm: Non-symmetric matrix - Total variation}. Assume that there exists $J\subseteq [n]$ such that $a_{ij}=1$ for all $i, j\in J$, where $|J|\geq cn$ for some $0<c < 1$. Then for any $k=o(\log n)$,
$$d_{TV}\l(\CZ_{k}(A\circ X),Z\r)\rightarrow 0\; \mbox{ as $n\rightarrow \infty$.}$$
\end{corollary}

\begin{proof}
	We observe that $b_n\leq n$ and from the assumption,  we have 
	\begin{align*}
		\sum_{I_k}a_{i_1i_2}^2\cdots a_{i_ki_1}^2\gtrsim \sum_{J^k}a_{i_1i_2}^2\cdots a_{i_ki_1}^2=c^{k}n^k.
	\end{align*}
	Now using the above estimates in Theorem \ref{thm: Non-symmetric matrix - Total variation} we have the result.
\end{proof}

\begin{remark} \label{rem: specially structred matrices} At first glance the assumption on $A$ in Corollary \ref{cor_iid_sparse} looks quite similar to \textit{super regularity} or \textit{broad connectivity} condition (see \cite{cook2016limiting}). For readers' convenience we give the definitions of super regularity and broad connectivity in  Section \ref{sec: Appendix}. But a close check reveals that the assumption in Corollary \ref{cor_iid_sparse}  is  weaker  than super regularity, and it is not comparable with broad connectivity. To see the difference from broad connectivity consider the following examples:
\begin{enumerate}[(i)]
\item Consider an $n\times n$ matrix $A$ of the form
\begin{align*}
A=\frac{1}{\sqrt{n}}\left[\begin{array}{ccccc}
0 & A_{12} & A_{13} & 0 & 0\\
0 & 0 & A_{23} & A_{24} & 0\\
0 & 0 & 0 & A_{34} & A_{35}\\
A_{41} & 0 & 0 & 0 & A_{45}\\
A_{51} & A_{52} & 0 & 0 & 0
\end{array}\right],
\end{align*}
where each $A_{ij}$ is $\lfloor \frac{n}{5}\rfloor\times \lfloor \frac{n}{5}\rfloor$ matrices with all $1$. The above matrix does not satisfy the assumption of Corollary \ref{cor_iid_sparse}, since there does not exist any index set $J\subset [n]$ such that $a_{ij}=1$ for all $i,j\in J$.  However, $A$ is a $(\delta, \nu)$-broadly connected matrix. Now we see that  $\CZ_{k}(A)\equiv 0$ for all $n$ and $k=1,2$, that is, $\CZ_{k}(A)\not\to Z$.

\item Now consider an $n\times n$ matrix $A$ of the form
\begin{align*}
A=\frac{1}{\sqrt{n}}\left[\begin{array}{ccccc}
0 & A_{12} & 0 & 0 & 0\\
0 & 0 & A_{23} & 0 & 0\\
0 & 0 & 0 & A_{34} & 0\\
0 & 0 & 0 & 0 & A_{45}\\
A_{51} & 0 & 0 & 0 & 0
\end{array}\right],
\end{align*}
where each $A_{ij}$ is $\lfloor \frac{n}{5}\rfloor\times \lfloor \frac{n}{5}\rfloor$ matrices with all $1$. In this case, it is easy to see that $\CZ_{k}(A\circ X)\equiv 0$ if $k$ is not a multiple of $5$. In other cases, $\CZ_{5l}(A\circ X)\to Z$ as $n\to\infty$. Here also, $A$ is a $(\delta, \nu)$-broadly connected matrix  but it does not satisfy the assumptions of Corollary \ref{cor_iid_sparse}.

\item Consider the following $n\times n$ matrix:
$$A=\left[\begin{array}{cc}
A_{11} & A_{12}\\
A_{21} & O
\end{array}
\right]
$$
where $A_{11}$ is $n/4 \times n/4$ matrix will all entries 1, $A_{12}$ is $n/4 \times 3n/4$ matrix will all entries 1, $A_{21}$ is $3n/4 \times n/4$ matrix will all entries 1 and $O$ is $3n/4 \times 3n/4$ matrix will all entries $0$. The matrix $A$ satisfies the assumptions of Corollary \ref{cor_iid_sparse} but it is not $(\delta, \nu)$-broadly connected. Observe that $\delta\leq 1/4$ and $A$ does not satisfy condition (iii) of the definition of broad connectivity (see Section \ref{sec: Appendix}).
For example, take $J=\{n/4 +1,\ldots, n\}$. Then $N_{A^T}(J)=\{i\in [n]: |N_A(i)\cap J|\geq \delta \frac{3n}{4}\}=\{1,2,\ldots,n/4\}.$ So $|N_{A^T}(J)|=n/4 \ngtr \min\{n,(1+\nu)\frac{3n}{4}\}$. 
However, $\CZ_{k}(A\circ X)$ will converge to $Z$ in total variation norm for  $k=o(\sqrt n)$.
\end{enumerate}
\end{remark}

In Remark \ref{rem: specially structred matrices} we showed  that the fluctuations of linear eigenvalue statistics are not Gaussian  in general for some specially structured variance profile matrices. But such specially structured matrix can arise in a Erd\"os-R\'enyi random graph. So an obvious question arise, whether our observation in Remark \ref{rem: specially structred matrices} contradicts  Corollary \ref{cor: for Erdos - Renyi graph} or not? It does not contradict, since in  Erd\"os-R\'enyi random graphs the probability that any special structure emerges in the graph (and hence in the adjacency matrix) is almost zero.

\begin{remark}\label{anti-diagonal}From Corollary \ref{cor_iid_sparse}, it also follows that the linear eigenvalue statistics of (symmetric and non-symmetric) anti-diagonal band  random  matrices  converge to Gaussian distribution in total variation norm when the band length is of  the order of $n$. By anti-diagonal random band matrix, we mean the matrices of the form $A\circ X$ with the following type of  $A$: 
$$
A=\left[\begin{array}{cccccccc}
0 & 0 & 0 & 0 & 1 & 1 & 1 & 1\\
0 & 0 & 0 & 1 & 1 & 1 & 1 & 0\\
0 & 0 & 1 & 1 & 1 & 1 & 0 & 0\\
0 & 1 & 1 & 1 & 1 & 0 & 0 & 0\\
1 & 1 & 1 & 1 & 0 & 0 & 0 & 0\\
1 & 1 & 1 & 0 & 0 & 0 & 0 & 1\\
1 & 1 & 0 & 0 & 0 & 0 & 1 & 1\\
1 & 0 & 0 & 0 & 0 & 1 & 1 & 1
\end{array}\right]
\;\mbox{ or }\;
A=\left[\begin{array}{cccccccc}
0 & 0 & 0 & 0 & 1 & 1 & 1 & 1\\
0 & 0 & 0 & 1 & 1 & 1 & 1 & 1\\
0 & 0 & 1 & 1 & 1 & 1 & 1 & 1\\
0 & 1 & 1 & 1 & 1 & 1 & 1 & 1\\
1 & 1 & 1 & 1 & 1 & 1 & 1 & 0\\
1 & 1 & 1 & 1 & 1 & 1 & 0 & 0\\
1 & 1 & 1 & 1 & 1 & 0 & 0 & 0\\
1 & 1 & 1 & 1 & 0 & 0 & 0 & 0
\end{array}\right].$$
\end{remark}

\section{Proof of Theorem \ref{thm: Non-symmetric matrix - Total variation}}\label{sec: Proof of main theorems about symmetric and non symmetric matrices}
In this section we give the proof of  Theorem \ref{thm: Non-symmetric matrix - Total variation}. The following lemmata will be used in the proof  of  Theorem \ref{thm: Non-symmetric matrix - Total variation}.

\subsection{Preliminary lemmas}\label{preliminary lemmas}
\begin{lemma}\label{lem:variance}
Suppose $A$ and $X$ are as in Theorem \ref{thm: Non-symmetric matrix - Total variation}. Then 	
\begin{align}\label{variance}
\var (\Tr[(A\circ X)^k])\geq \sum_{I_k}a_{i_1i_2}^2\cdots a_{i_ki_1}^2,
\end{align}
where $I_k=\{(i_1,\ldots,i_k)\in [n]^k \;:\; i_p \neq i_q,\;\mbox{for all}\;p\neq q\}$. 
\end{lemma}

\begin{proof} For a positive integer $k$, we have 
	\begin{align}\label{eqn:var1}
	&\var\l(\Tr(A\circ X)^k\r)=\E[\l(\Tr(A\circ X)^k\r)^2]-\E[\l(\Tr(A\circ X)^k\r)]^2\\
	=&\sum_{I_k',J_k'}\big(\E[X_{i_1i_2}\cdots X_{i_ki_1}X_{j_1j_2}\cdots X_{j_kj_1}]- \notag \\
	& \hspace{.5in}\E[X_{i_1i_2}\cdots X_{i_ki_1}]\E[X_{j_1j_2}\cdots X_{j_kj_1}]\big)a_{i_1i_2}\cdots a_{i_ki_1}a_{j_1j_2}\cdots a_{j_kj_1},\notag
	\end{align}
where $I_k',J_k'\in [n]^k$.	
	Note that, we have 
	\begin{align}\label{eqn:var2}
	&\E[X_{i_1i_2}\cdots X_{i_ki_1}X_{j_1j_2}\cdots X_{j_kj_1}]-\E[X_{i_1i_2}\cdots X_{i_ki_1}]\E[X_{j_1j_2}\cdots X_{j_kj_1}]\nonumber
	\\=&\prod_{i,j}\E[X_{ij}^{\alpha_{ij}+\beta_{ij}}]-\prod_{ij}\E[X_{ij}^{\alpha_{ij}}]
	\E[X_{ij}^{\beta_{ij}}]\ge 0,
	\end{align}
	for some $\alpha_{ij},\beta_{ij}\in \{0,1,2,\ldots\}$ and $i,j\in \{1,2,\ldots, n\}$.	 The last inequality follows from the fact that $\E[X_{ij}^{\alpha_{ij}+\beta_{ij}}]\ge \E[X_{ij}^{\alpha_{ij}}]
	\E[X_{ij}^{\beta_{ij}}]$, as $X_{ij}$ are symmetric random variables. Indeed, if $\alpha_{ij}+\beta_{ij}$ is odd then $\E[X_{ij}^{\alpha_{ij}+\beta_{ij}}]=0$ and one of $\E[X_{ij}^{\alpha_{ij}}]$ and $\E[X_{ij}^{\beta_{ij}}]$ is zero, as a symmetric random variable has odd moments zero. Suppose $\alpha_{ij}+\beta_{ij}$ is even, and both $\alpha_{ij},\beta_{ij}$ are odd, then $\E[X_{ij}^{\alpha_{ij}}]=\E[X_{ij}^{\beta_{ij}}]=0$. Finally if both $\alpha_{ij},\beta_{ij}$ are even, then  by H\"older inequality we have
	\begin{align*}
	\E[X_{ij}^{\alpha_{ij}}]\le (\E[X_{ij}^{\alpha_{ij}+\beta_{ij}}])^{\frac{\alpha_{ij}}{\alpha_{ij}+\beta_{ij}}}\; \mbox{and } \E[X_{ij}^{\beta_{ij}}]\le (\E[X_{ij}^{\alpha_{ij}+\beta_{ij}}])^{\frac{\beta_{ij}}{\alpha_{ij}+\beta_{ij}}}.
	\end{align*}
	Therefore  $\E[X_{ij}^{\alpha_{ij}+\beta_{ij}}]\ge \E[X_{ij}^{\alpha_{ij}}]
	\E[X_{ij}^{\beta_{ij}}]$. Since $a_{ij}$ are non-negative,  from \eqref{eqn:var1} and \eqref{eqn:var2} we get
	\begin{align*}
	\var\l(\Tr(A\circ X)^k\r)&\ge \sum_{I_k}\E[X_{i_1i_2}^2\cdots X_{i_ki_1}^2]a_{i_1i_2}^2\cdots a_{i_ki_1}^2=\sum_{I_k}a_{i_1i_2}^2\cdots a_{i_ki_1}^2.
	\end{align*}
	In the last line we used the fact that  $E[X_{i_1i_2}\cdots X_{i_ki_1}]=0$ and $\E[X_{i_1i_2}^2\cdots X_{i_ki_1}^2]=1$ for $(i_1,\ldots, i_k)\in I_k$, as $X_{ij}$ are i.i.d. random variables with mean zero and variance one. This completes the proof of the lemma.
\end{proof}

\begin{result}\cite[Theorem 1.2]{chafai2018poincar} \label{Result: Function of gaussian satisfies entropy inequality} Let $\mu$ be a probability measure on $\R^{n}$ such that $\mu(dx)\propto \exp\left\{\sum_{i=1}^{n}V(x_{i})\right\}$. Assume that $V:\R\to\R$ is a $\kappa$-convex function i.e., $V(x)-\kappa x^{2}/2$ is a convex function for some fixed $\kappa >0$. Then for any $g\in H^{1}(\mu)$,
	\begin{align*}
	\text{Ent}(g^{2})\leq \frac{2}{\kappa}\int_{\R^{n}}|\nabla g|^{2}\;d\mu,
	\end{align*}
	where $\text{Ent}(g)$ is defined by
	\begin{align*}
	\text{Ent}(g):=\int_{\R^{n}} g\log g\;d\mu-\left(\int_{\R^{n}}g\;d\mu\right)\log \left(\int_{\R^{n}}g\;d\mu\right)
	\end{align*}
	for any $g>0$.
	\end{result}

\begin{result}\cite[Proposition 2.3]{ledoux1999concentration}\label{Lemma: Entropy inequality implies the tail bound} Let $\{X_1,X_2, \ldots,X_k\}$ be a collection of random variables with $\E[e^{\lambda X_i}]<\infty$ for all $\lambda\in\R$ and 
	\begin{align*}
	\text{Ent}(e^{X_i})\leq c^{2}\E[e^{X_i}]  \ \mbox{ for }\ 1\leq i \leq k.
	\end{align*} 
	Suppose $F:\mathbb R^k \to \R$ be  a Lipschitz function with $\|F\|_{\text{Lip}}\leq 1$. Then
	\begin{align*}
	\P\left(F(X_1,\ldots,X_k)\geq \E[F(X_1,\ldots,X_k)]+t\right)\leq e^{-t^{2}/c^{2}},
	\end{align*}
	for every $t\geq 0$.
\end{result}

\begin{result}\cite[Corollary 3.5]{bandeira2016sharp}\label{ft:norm}
	Let $X$ be an $n\times n$ symmetric matrix with $X_{ij}=\xi_{ij}b_{ij}$,
	where $\{\xi_{ij}: i\ge j\}$ are independent centred random variables and $\{b_{ij} :  i\ge j\}$
	are given scalars. If $\E[\xi_{ij}^{2p}]^{\frac{1}{2p}}\le K p^{ \frac{\beta}{2}}$ for some $K,\beta>0$ and all $p,i,j$ then
	\begin{align*}
	\E\|X\|\lesssim \max_{i}\sqrt{\sum_{j}b_{ij}^{2}} +\max_{ij}|b_{ij}| (\log n)^{\frac{\max\{\beta, 1\}}{2}}.
	\end{align*}
	The universal constant in the above inequality depends on $\beta$ only.
\end{result}

\begin{lemma}\label{Lemma: Norm of the Matrix} Let $X=(X_{ij})_{n\times n}$ be a random matrix, where $\{(X_{ij},X_{ji}): 1\leq i,j\leq n\}$ be a collection of independent centered random vectors with $\E[X_{ij}^2]=\E[X_{ji}^2]=1$, $\E[X_{ij}X_{ij}]=\rho\in \{0, 1\}$ and $X_{ij}\in\CL(c_{1},c_{2})$, for some fixed $c_{1}>0,c_{2}>0$, $\forall\; i,j$. Let $A=(a_{ij})_{n\times n}$ be a fixed deterministic matrix such that $a_{ij}\in[0,1],$ and in addition, $a_{ij}=a_{ji}$ for all $i,j$ if $\rho=1$. Then 
	\begin{align}\label{Equaltion: Norm of the matrix}
	\P\left(\|A\circ X\|>K\max_{i,j}\l(\sqrt{\sum_{k}a_{ik}^{2}}+\sqrt{\sum_{k}a_{kj}^{2}}+|a_{ij}|\sqrt{\log n}\r)+t\right)
	\leq e^{-t^{2}/c_{1}^{2}},
	\end{align}
	where $K>0$ is an universal constant. In particular, with probability one
	\begin{align*}
	\|A\circ X\|\leq K\max_{i}\sqrt{\sum_{j}a_{ij}^{2}}+ K\max_{j}\sqrt{\sum_{i}a_{ij}^{2}}+(K\max_{ij}|a_{ij}|+\sqrt2 c_1)\sqrt{\log n}
	\end{align*}	
	for all but finitely many $n$.
\end{lemma}

\begin{proof}
	This proof is based on ~\cite[Corollary 3.5]{bandeira2016sharp}. First of all if $\rho=1$, then both $A$ and $X$ are symmetric matrices. In that case, the result \ref{ft:norm} gives bound on $\E\|A\circ X\|$. If $\rho=0$,  let us write the matrix $A\circ X$ in the following way.
	
	\begin{align*}
	A\circ X&=\frac{1}{2}[(A\circ X)+(A\circ X)^{t}]+\frac{1}{2}[(A\circ X)-(A\circ X)^{t}]\\
	&:=(B_{n}^{+}\circ X^{+})+(B_{n}^{-}\circ X^{-}),
	\end{align*}
	where 
	\begin{align*}
	(B_{n}^{+})_{ij}&:=\sqrt{a_{ij}^{2}+a_{ji}^{2}}, & (B_{n}^{-})_{ij}&:=\sqrt{a_{ij}^{2}+a_{ji}^{2}},\\
	(X^{+})_{ij}&:=(a_{ij}X_{ij}+ a_{ji}X_{ji})/(2B_{n}^{+})_{ij},& (X^{-})_{ij}&:=(a_{ij}X_{ij}- a_{ji}X_{ji})/(2B_{n}^{-})_{ij}.
	\end{align*}

	In the above definitions, we use the convention that $(X^{+})_{ij}=0=(X^{-})_{ij}$ if both $a_{ij}=0=a_{ji}$. Since $X_{ij}\in \CL(c_{1},c_{2})$, we may write 
	\begin{align*}
	X_{ij}=u(0)+u'(\theta_{ij})Z_{ij},
	\end{align*}
	where $Z_{ij}\sim N(0,1)$, and $\theta_{ij}$ lies in between $0$ and $Z_{ij}$. Since $X_{ij}\in\CL(c_{1},c_{2})$, we have $|u'(\theta_{Z})|\leq c_{1}$. Therefore for any $p\in \mathbb{N}$,
	\begin{align*}
	\E[X_{ij}^{2p}]^{1/2p}\leq |u(0)|+c_{1}\E[Z^{2p}]^{1/2p}\leq |u(0)|+c_{1}\sqrt{p}\leq C\sqrt{p},
	\end{align*}
	where $C=|u(0)|+c_{1}$. Consequently by Minkowski's inequality, we have 
	\begin{align*}
	\E[(X^{\nu})_{ij}^{2p}]^{1/2p}&\leq \frac{C}{|2(B_{n}^{\nu})_{ij}|}[|a_{ij}|+|a_{ji}|]\sqrt{p}\leq C\sqrt{p},\;\;\;\nu\in \{+,-\},\\
	\end{align*}
	
	where the last inequality follows from the fact that $|a_{ij}|+|a_{ji}|\leq |2(B_{n}^{\nu})_{ij}|$ for $\nu\in\{+,-\}$.  Hence using result \ref{ft:norm}, we conclude that for $\nu\in \{+,-\}$,
	\begin{align*}
	\E\|B_{n}^{\nu}\circ X^{\nu}\|&\leq K\left[\max_{i}\sqrt{\sum_{j}(B_{n}^{\nu})_{ij}^{2}}+\max_{ij}|(B_{n}^{\nu})_{ij}|\sqrt{\log n}\right]\\
	&\leq K\max_{i}\sqrt{\sum_{j}a_{ij}^{2}}+K\max_{j}\sqrt{\sum_{i}a_{ij}^{2}}+K\max_{ij}|a_{ij}|\sqrt{\log n},
	\end{align*}
	where $K>0$ is a universal constant.
	
To prove the almost sure bound, we invoke the results \ref{Result: Function of gaussian satisfies entropy inequality}, and \ref{Lemma: Entropy inequality implies the tail bound}.  If $Z\sim N(0,1)$ and $u:\R\to\R$ is a differentiable function with $|u'|\leq c_{1}$ uniformly, then taking $\kappa=1/2$ in result \ref{Result: Function of gaussian satisfies entropy inequality}
		\begin{align*}
		\text{Ent}\left(e^{u(Z)}\right)\leq 4\E\left[\left|\frac{u'(Z)}{2}\right|^{2}e^{u(Z)}\right]\leq c_{1}^{2}\E\left[e^{u(Z)}\right].
		\end{align*}

	Now since $\|M\|\leq \sqrt{\sum_{i,j}M_{ij}^{2}}$, $M\mapsto\|M\|$ is a Lipschitz function with Lipschitz constant equal to $1$. Therefore, the estimate \eqref{Equaltion: Norm of the matrix} follows from the result \ref{Lemma: Entropy inequality implies the tail bound}.
\end{proof}

\subsection{Proof of the Theorem \ref{thm: Non-symmetric matrix - Total variation}}\label{proof of main result}
We first introduce some notations which will be used in the proof of Theorem \ref{thm: Non-symmetric matrix - Total variation}.

\begin{align*}
f(z)&=\sum_{k=0}^{\infty}d_kz^k, \; f_1(z)=\sum_{k=1}^{\infty}k|d_k|z^{k-1}, \; f_2(z)=\sum_{k=2}^{\infty}k(k-1)|d_k|z^{k-2},\\
\mathcal L_f(X)&=\sum_{i=1}^{n}f(\lambda_i), \mbox{where $\lambda_1,\ldots, \lambda_n$ are the eigenvalues of $X$}.
\end{align*}
Let $\mathcal J$ be a finite index set. Define
\begin{align*}
\mathcal R &= \{\alpha \in \C^{\mathcal J}\; : \; \sum_{u\in \mathcal J}|\alpha_u|^2=1 \},\;\;\;
\mathcal S = \{ \beta\in \C^{n\times n} \; : \; \sum_{i,j=1}^n |\beta_{i,j}|^2=1 \}.
\end{align*}
Let $H(x)=(h_{ij})$ be an $n\times n$ matrix, where $h_{ij}:\R^{\mathcal J}\to \C$ are $C^2$ maps for $1\le i,j\le n$. Define three functions $\gamma_0, \gamma_1$ and $\gamma_2$ on $\R^{\mathcal J}$ as follows
\begin{align*}
 \gamma_0(x)&=\sup_{u\in \mathcal J,\|B\|=1}\l |\Tr\left(B\frac{\partial H}{\partial x_u}\right)\r|,\\
\gamma_1(x)&=\sup_{\alpha\in \mathcal R,\beta \in \mathcal S}\l | \sum_{u\in \mathcal J}\sum_{i,j}\alpha_u\beta_{ij}\frac{\partial h_{ij}}{\partial x_u}\r|, \\
\gamma_2(x)&=\sup_{\alpha,\alpha' \in \mathcal R,\beta \in \mathcal S}\l| \sum_{u,v\in \mathcal J}\sum_{i,j}\alpha_u\alpha_u'\beta_{ij}\frac{\partial^2h_{ij}}{\partial x_u\partial x_v}\r|
\end{align*}
Let $\lambda(x)=\|H(x)\|, \; \; r(x)=\mbox{rank}(H(x))$. Define a few more functions
\begin{align*}
\eta_0(x)&=\gamma_0(x)f_1(\lambda(x)),\;\; \eta_1(x)=\gamma_1(x)f_1(\lambda(x))\sqrt{r(x)},\\
\eta_{2}(x)&=\gamma_2(x)f_1(\lambda(x))\sqrt{(r(x))}+\gamma_1(x)^2f_2(\lambda(x)),\\
\kappa_0&=\l(\E\eta_0(x)^2\eta_1(x)^2\r)^{\frac{1}{2}},\ \kappa_1=\l(\E \eta_1(x)^4\r)^{\frac{1}{4}},\ \kappa_2 =\l(\E\eta_2(x)^4\r)^{\frac{1}{4}}.
\end{align*}
The following result from \cite{chatterjee2009fluctuations} is  the main ingredient of our proof of Theorem \ref{thm: Non-symmetric matrix - Total variation}.

\begin{result}\cite[Theorem 3.1]{chatterjee2009fluctuations} \label{re:ch}
Let all notation be as above. Suppose $W=\mbox{Re}\Tr f(H(x))$ has finite fourth moment and let $\sigma^2=\var(W). $ Let $Z$ be a random variable with the same mean and variance as $W$. Then
$$
d_{TV}(W,Z)\le \frac{2\sqrt 5(c_1c_2\kappa_0+c_1^3\kappa_1\kappa_2)}{\sigma^2}.
$$
\end{result}

In our setting, $H_n=A\circ X$ and $\mathcal J=\{(i,j): 1\le i,j\le n\}$ when $X$ is an iid matrix. We will use Result \ref{re:ch} to prove the theorem and for that we first estimate $ \kappa_0,\kappa_1$ and $\kappa_2$ for $A\circ X$. Note
\begin{align*}
\gamma_0(x)&=\sup_{u\in \mathcal J,\|B\|=1}\l |\Tr\left(B\frac{\partial (A\circ X)}{\partial x_u}\right)\r|\le \max |a_{ij}|,\\
\gamma_1(x)&=\sup_{\alpha\in \mathcal R,\beta \in \mathcal S}\l | \sum_{u\in \mathcal J}\sum_{I_2}\alpha_u\beta_{ij}\frac{\partial (a_{ij}x_{ij)}}{\partial x_u}\r|\le \max |a_{ij}|, \\
\gamma_2(x)&=\sup_{\alpha,\alpha' \in \mathcal R,\beta \in \mathcal S}\l| \sum_{u,v\in \mathcal J}\sum_{i,j}\alpha_u\alpha_u'\beta_{ij}\frac{\partial^2(a_{ij}x_{ij})}{\partial x_u\partial x_v}\r|=0.
\end{align*}
Let $f(z)=P_k(z)=\sum_{i=0}^{n}c_iz^i$, where $c_0,\ldots, c_k\in \R$ with $c_k\neq 0$.   Then, for $\lambda \ge 1$, 
\begin{align}\label{eqn:ff1f2}
	f(\lambda)&=\lambda^k\sum_{i=0}^{k}c_i\lambda^{i-k}\le k\lambda^k \max\{|c_0|,\ldots, |c_k|\}\lesssim k\lambda^k,\nonumber\\
	f_1(\lambda)&=k\lambda^{k-1}\sum_{i=1}^{k}|c_i|\frac{i}{k}\lambda^{i-k}\lesssim k^2\lambda^{k-1},
	\\f_2(\lambda)&=k(k-1)\lambda^{k-2}\sum_{i=2}^{k}|c_i|\frac{i(i-1)}{k(k-1)}\lambda^{i-k}\lesssim k^3\lambda^{k-2}.\nonumber
\end{align}
Then we have 
\begin{align*}
\eta_0(x)&=\gamma_0(x)f_1(\lambda_n(x))\lesssim \max |a_{ij}|k^2 \lambda_n(x)^{k-1},\\
\eta_1(x)&=\gamma_1(x)f_1(\lambda_n(x))\sqrt{r_n(x)}\lesssim \max |a_{ij}| k^2 \lambda_n(x)^{k-1}\sqrt {r_n(x)},\\
\eta_{2}(x)&=\gamma_2(x)f_1(\lambda_n(x))\sqrt{(r_n(x))}+\gamma_1(x)^2f_2(\lambda(x))\lesssim (
\max |a_{ij}|)^2k^3\lambda_n(x)^{k-2},
\end{align*}
where $\lambda_n=\|A\circ X\|$. Again, note that, Lemma \ref{Lemma: Norm of the Matrix}  implies that $\lambda_n\le \sqrt{b_n}$. Since  $\mbox{rank}(A\circ X)= r_{n}\le n$ almost surely
\begin{align*}
\kappa_0&=\l(\E\eta_0(x)^2\eta_1(x)^2\r)^{\frac{1}{2}}\lesssim (\max |a_{ij}|)^2k^4b_{n}^{k-1} \sqrt {n} ,\;\;\mbox{almost surely},\\
\kappa_1&=\l(\E \eta_1(x)^4\r)^{\frac{1}{4}}\lesssim  (\max |a_{ij}| )k^2 b_n^{\frac{k-1}{2}}\sqrt {n},\;\;\mbox{almost surely},\\
\kappa_2 &=\l(\E\eta_2(x)^4\r)^{\frac{1}{4}}\lesssim (\max|a_{ij}|)^2k^3b_n^{\frac{k-2}{2}}, \;\;\mbox{almost surely}.
\end{align*}
Therefore from Result \ref{re:ch} and Lemma \ref{lem:variance}, we have 
\begin{align}\label{eqn:bound}
d_{TV}(W_n,Z_n)\lesssim \frac{ (\max |a_{ij}|)^2k^5 b_{n}^{k-1} \sqrt {n}. }{\sigma_n^2}
\end{align}
This completes proof of the Theorem \ref{thm: Non-symmetric matrix - Total variation} for non-symmetric matrix. 

For the symmetric case, that is, when both $X$ and $A$ are symmetric,  proof goes in the similar line. In this case  also $\|A\circ X\|\lesssim \sqrt {b_n}$ a.s. and it follows immediately from \cite[Corollary 3.12]{bandeira2016sharp}. Here we skip the details.

\section{Appendix}\label{sec: Appendix}
Here we give the definitions of broadly connected and super regular variance profiles mentioned in Remark \ref{rem: specially structred matrices}. For further information and results  on these variance profiles, we refer to \cite{cook2016limiting}. Let $A$ be an $n\times m$ deterministic matrix with non-negative entries $a_{ij}$ and  define 
\begin{align*}
\mathcal N_A(i)&=\{j\in [m]\; :\; a_{ij}>0\},\\
e_A(I,J)&=\l|\{(i,j)\in [n]\times [m]\; : \; a_{ij}>0 \}\r|.
\end{align*}
\begin{definition}[Broad connectivity]
	Let $A=(a_{ij})$ be an $n\times m $ matrix with non-negative entries. For $I\subset [n]$ and $\delta\in (0,1),$ define the set $\delta$-broadly connected neighbours of $I$ as 
	$$
	\mathcal N_A^{(\delta)}(I)=\{j\in [m]\; : \; |\mathcal N_{A^T}(j)\cap I|\ge \delta |I|\}.
	$$
	For $\delta, \nu\in (0,1)$, we say that $A$ is $(\delta, \nu)$-broadly connected if 
	\begin{enumerate}
		\item[(i)] $|\mathcal N_{A}(i)|\ge \delta m$, for all $i\in [n]$,
		
		\item[(ii)] $|\mathcal N_{A^T}(j)|\ge \delta n$ for all $j\in [m]$,
		
		\item[(iii)] $|\mathcal N_{A^T}^{\delta}(J)|\ge \min(n, (1+\nu )|J|)$ for $|J|\subset [m]$.
	\end{enumerate}
\end{definition}

\begin{definition}[Super regularity]
	Let $A$ be an $n\times m $ matrix with non-negative entries, for $\delta, \epsilon\in (0,1)$, we say that $A$ is $(\delta, \epsilon)$-super regular if the following hold:
	\begin{enumerate}
		\item[(i)] $|\mathcal N_A(i)|\ge \delta m$ for all $i\in [n]$,
		
		\item[(ii)] $|\mathcal N_{A^T}(j)|\ge \delta n$ for all $j\in [m]$,
		
		\item[(iii)] $e_A(I,J)\ge \delta |I||J|$ for all $I\subset [n], J\subset [m]$ such that $|I|\ge \epsilon n$ and $|J|\ge \epsilon m$.
	\end{enumerate}
\end{definition}

\bibliographystyle{amsplain}

\end{document}